\documentclass[11pt, reqno]{amsart}

\usepackage{amsmath,amsfonts,amsthm,amssymb,color,hyperref}

\usepackage{mathrsfs,soul}

\usepackage[T1]{fontenc}

\usepackage{graphicx}
\usepackage{subfigure} 

\makeatletter
     \def\section{\@startsection{section}{1}%
     \z@{.7\linespacing\@plus\linespacing}{.5\linespacing}%
     {\bfseries
     \centering
     }}
     \def\@secnumfont{\bfseries}
     \makeatother

\usepackage{graphicx}

\newcommand{\R}{\mathbb R}
\newcommand{\RR}{\mathbb R}

\newcommand{\E}{\mathbb E}

\newcommand{\ep}{\varepsilon}

\newtheorem{theorem}{Theorem}[section]
\newtheorem{lemma}[theorem]{Lemma}

\theoremstyle{definition}
\newtheorem{definition}[theorem]{Definition}
\theoremstyle{remark}

\numberwithin{equation}{section}
\begin{document}

\title[Asymptotic properties]{Asymptotic properties of the stochastic heat equation in large times}

 \date{\today}

\author{Arturo Kohatsu-Higa* \and David Nualart}

\address{Corresponding Author: Arturo Kohatsu-Higa: 
	Department of Mathematical Sciences Ritsumeikan University 1-1-1 Nojihigashi, Kusatsu, Shiga, 525-8577, Japan.}
\email{khts00@fc.ritsumei.ac.jp}
\thanks{The research of the second author was supported by KAKENHI grants 16K05215 and 16H03642.}

\address{David Nualart: University of Kansas, Mathematics department, Snow Hall, 1460 Jayhawk blvd, Lawrence, KS 66045-7594, United States} \email{nualart@ku.edu} \thanks{D. Nualart is supported by NSF Grant DMS 1811181.}

\begin{abstract}
 In this article, we study the asymptotic behavior of the stochastic heat equation for large times.

\end{abstract}
 \maketitle

  \section{Introduction and main results}

Suppose that $W=\{W(t,x), t\ge 0, x\in \RR\}$ is a two-parameter Wiener process.  That is, $W$ is a zero-mean Gaussian process with covariance function given by
\[
\E (W(t,x) W(s,y) ) = (s\wedge t) ( |x| \wedge |y|)\mathbf{1}_{\{xy >0\}}.
\]
Consider the stochastic heat equation
\begin{equation} \label{ecu1}
\frac {\partial u }{\partial t} = \frac 12 \frac {\partial ^2u }{\partial x^2} + \varphi(W(t,x))  \frac { \partial ^2 W} {\partial t\partial x}, \quad x\in \RR, t\ge 0,
\end{equation}
where $\varphi: \RR \rightarrow \RR$ is a given Borel  measurable  function such that for each $t \ge 0$ and $x\in \RR$,
\begin{equation} \label{ecu2}
\int_0^t \int_{\RR^2}  p^2_{t-s} (x-y) \varphi^2(z) p_{s|y|} (z) dydzds <\infty.
\end{equation}
Along the paper $p_t(x)$ denotes the  one-dimensional heat kernel, that is, $p_t(x)= (2\pi t)^{-1/2} e^{ -x^2/2t}$ for $t>0$ and $x\in \RR$.  The mild solution to equation (\ref{ecu1}) with initial condition $ u(0,x)=0 $ is given by
\[
u(t,x)= \int_0^t \int_{\RR}  p_{t-s} (x-y)  \varphi(W(s,y)) W(ds,dy),
\]
where the stochastic integral is well defined in view of condition (\ref{ecu2}).

We are interested in the asymptotic behavior as $t\rightarrow \infty$  of $u(t,x)$ for $x\in \RR$ fixed.  Notice first that in the particular case where $\varphi(x)\equiv c$, then $u(t,x)$ is a centered  Gaussian random variable with variance
\[
\E(u(t,x)^2)= c^2 \int_0^t \int_{\RR}  p^2_{t-s} (x-y) dyds = c^2 \int_0^t  p_{2(t-s)} (0) ds = \frac {c^2}{\sqrt{\pi} }    \sqrt{t}.
\]
Therefore $t^{-\frac 14} u(t,x)$ has the law $N(0, c^2/ \sqrt{\pi})$.
In the general case, using the change of variables $s \rightarrow ts$  and $y \rightarrow \sqrt{t }y$, we can write
\begin{align}
u(t,x)&=   \int_0^1 \int_{\RR}  p_{t(1-s)} (x-y)  \varphi(W(ts,y)) W(tds,dy) \nonumber\\
&= \frac 1{ \sqrt{t}}   \int_0^1 \int_{\RR}  p_{1-s} (\frac {x-y}{\sqrt{t}})  \varphi(W(ts,y)) W(tds,dy) \nonumber \\
&= \frac 1{ \sqrt{t}}   \int_0^1 \int_{\RR}  p_{1-s} (\frac {x}{\sqrt{t}} -y )  \varphi(W(ts, \sqrt{t} y)) W(tds,\sqrt{t}dy).
\label{eq:1}
\end{align}
By the scaling properties of the two-parameter Wiener process it follows that $u(t,x)$ has the same law as
\begin{equation} \label{ecu3}
\widetilde{u}(t,x)= t^{1/4}    \int_0^1 \int_{\RR}  p_{1-s} (\frac {x}{\sqrt{t}} -y )  \varphi( t^{3/4} W(s, y)) W(ds,dy).
\end{equation}
The asymptotic behavior of $u(t,x)$ will depend on the properties of the function $\varphi$. We will consider three  classes of functions for  which different behaviors appear.
We are going to use the following notion of convergence, which is stronger that the convergence in distribution (see, for instance,  \cite[Chapter 4]{JacSh}).
\begin{definition} \label{d:stable} 
  Let $\{F_{n}\}$ be a sequence of random variables  defined on a probability space $(\Omega, \mathcal{F}, \mathbb{P})$. Let $F$ be a random variable defined on some extended probability space $(\Omega', \mathcal{F}', \mathbb{P}')$.  We say that $F_{n}$ {\it converges  stably} to $F$, written $F_n \stackrel {\rm stably} {\longrightarrow} F$, if
\begin{equation}
\underset{n \rightarrow \infty}{\rm lim}\mathbb{E}\left[Ge^{i \lambda F_{n}} \right] = \mathbb{E}'\left[Ge^{i \lambda  F} \right],\label{e:stable}
\end{equation}
for every $\lambda \in \R$ and every bounded $\mathcal{F}$--measurable random variable $G$. 
\end{definition}

The first theorem deals with the case where $\varphi$ is an homogeneous type function. 

 \begin{theorem}
 	\label{thm1b}
 Suppose that $ \varphi:\R\rightarrow \R$,   is a measurable and bounded on compacts function such that $\lim _{x  \rightarrow \pm\infty} |x|^{-\alpha}\varphi(x) =c_\pm$  for some constants $c_+, c_-$ and $ \alpha\geq 0 $.  Then,  as $t\rightarrow \infty$, 
 \begin{align*}
 & t^{-\frac {3\alpha +1}4} u(t,x)   \stackrel  {\rm stably}  \longrightarrow   c_- \int_0^1 \int_{\RR}  p_{1-s} (y ) |\widehat{W}(s, y)|^\alpha \mathbf{1}_{\{\widehat{W}(s,y) <0\}}  \widehat{W}(ds,dy)\\
 & \qquad \qquad \qquad \qquad  +c_+\int_0^1 \int_{\RR}  p_{1-s} (y ) |\widehat{W}(s, y)|^\alpha \mathbf{1}_{\{\widehat{W}(s,y) >0\}}  \widehat{W}(ds,dy)=:X.
 \end{align*}
Here  $\widehat{W}$ is a two-parameter Wiener process independent of $W$.
 \end{theorem}
 
 Note that in the case that $ c_+=c_-$ and $ \alpha=0 $ then the limit is Gaussian. Note that one may also consider the case 
 $\lim _{x  \rightarrow -\infty} |x|^{\pm\alpha_\pm}\varphi(x) =c_\pm$ for some constants $c_+, c_-$, $ \alpha_+,\alpha_-\geq 0 $. In this case the renormalization factor is $ t^{-\frac {3(\alpha_+\vee\alpha_- )+1}4} $ and the limit will only have contributions from the largest $ \alpha_i=\alpha_+\vee\alpha_- $.
 
 In the second theorem we consider the case were $\varphi$ satisfies some integrability properties with respect to the Lebesgue measure on $\RR$.
 The limit involves a weighted local time of the two-parameter Wiener process and the proof has been inspired by the work of Nualart and Xu  \cite{NuXu} on the central limit theorem for an additive functional of the fractional Brownian motion.

 \begin{theorem}  \label{thm2} Suppose that $\varphi \in L^2(\RR) \cap L^p(\RR)$ for some $p<2$.  Then,  as $t\rightarrow \infty$, 
\[
t^{\frac 18} u(t,x)   \stackrel  {\rm stably}  \longrightarrow   Z \left( \int_0^1 \int_{\RR}  p^2_{1-s} (y ) \delta_0(\widehat{W}(s, y)) dyds \right)^{\frac 12} \| \varphi \|_{L^2(\RR)} ,
\]
where   $\widehat{W}$ is a two-parameter Wiener process, $Z$  is a $N(0,1)$ random variable and  $\widehat{W}$, $W$ and $Z$ are independent.
 \end{theorem}
  
 In Theorem \ref{thm2},   $L_{0,1}:=\int_0^1 \int_{\RR}  p^2_{1-s} (y ) \delta_0(\widehat{W}(s, y)) dyds$ is a weighted local time of the random field   $\widehat{W}$, that can be defined  (see Lemma \ref{lem1} below) as the limit in $L^2(\Omega)$ of 
 \[
 L^\varepsilon_{0,1}:=\int_0^1 \int_{\RR}  p^2_{1-s} (y ) p_\varepsilon(\widehat{W}(s, y)) dyds,
 \]
 as  $\ep$ tends to zero.

 The paper  is organized as follows. Section 2 contains the proofs of the above theorems and in Section 3 we discuss an extension of these results in the case where we consider also space averages on an interval $[-R,R]$ and both $R$ and $t$ tend to infinity. A brief discussion of the case of nonlinear equations is provided in Section 4.
 
 \section{Proofs}
 \noindent
 {\it Proof of Theorem \ref{thm1b}:}  We know that $  t^{-\frac { 3\alpha +1}4}    u(t,x)$ has the same law as
 \[
 v(t,x ):=  t^{-\frac { 3\alpha }4}  \int_0^1 \int_{\RR}  p_{1-s} (\frac {x}{\sqrt{t}} -y ) \varphi(t^{3/4}W(s,y)) W(ds,dy).
 \]
 
 We divide the study of $ v $ into two parts according to the boundedness on compacts property for $ \varphi $. In fact, for any compact $ K $ consider $ \varphi_K(x)= \varphi(x)\mathbf{1}_{x\in K}$. Then we will prove that $ v_K(x)\to 0 $ in $ L^2(\Omega) $ where
 \begin{align*}
 v_K(t,x ):=  t^{-\frac { 3\alpha }4}  \int_0^1 \int_{\RR}  p_{1-s} (\frac {x}{\sqrt{t}} -y ) \varphi_K(t^{3/4}W(s,y)) W(ds,dy).
 \end{align*}
 
 In fact, as $ \varphi_K $ is bounded by a constant, say $ M $, we have 
 \begin{align*}
 \E[v^2_K(t,x)]\leq M^2t^{-\frac{3\alpha}2}\int_0^1 \frac 1 { \sqrt{2\pi(1-s)}} \int_{\RR}  p_{1-s} (\frac {x}{\sqrt{t}} -y ) \mathbb{P}(t^{3/4}W(s,y)\in K) dsdy.
 \end{align*}
 The above quantity clearly converges to zero  as $t\rightarrow \infty$, if one considers separately the cases $ \alpha>0 $ and $ \alpha=0 $.
 
 Given the above result, we can assume without loss of generality that $ \varphi(x)=f(x)|x|^\alpha $ with a bounded measurable function $ f $ such that $ \lim_{x\to\pm\infty}f(x)=c_\pm $.

 For this,   we fix $t_0>0$ and compute the conditional characteristic function of $t^{-\frac { 3\alpha +1}4}    u(t,x)$ given $\mathcal{F}_{t_0}$,
 where $ t>t_0 $ and $\{ \mathcal{F}_t, t\ge 0\}$ denotes the natural filtration of the two-parameter Wiener process $W$ used in the definition of $ u $\footnote{That is, $ \mathcal{F}_t $ is generated by $ W(s,x), s\leq t, x\in \R $.} . For any $\lambda \in \R$, we have
 \begin{align*}
 \E \left[ e^{i\lambda  t^{-\frac { 3\alpha +1}4}    u(t,x) } | \mathcal{F}_{t_0} \right] &
 =e^{i\lambda  t^{-\frac { 3\alpha +1}4}     \int_0 ^{t_0} \int_{\RR} p_{t-s} (x-y) \varphi(W(s,y)) W(ds,dy) } \\
 & \times   \E \left[   e^{i\lambda  t^{-\frac { 3\alpha +1}4}    \int_{t_0} ^{t} \int_{\RR} p_{t-s} (x-y)  \varphi(W(s,y))  W(ds,dy ) } | \mathcal{F}_{t_0} \right]\\
 & =: e^{i\lambda A_t} \times B_t.
 \end{align*}
 It is easy to show that $\lim_{t\rightarrow \infty}  A_t =0$ in $L^2(\Omega)$ as $ t\rightarrow\infty $. 
 In fact, using the rescaling properties we obtain that as $ t\to\infty $
 \begin{align*}
 & t^{-\frac { 3\alpha +1}2}   \E\left[  \int_0 ^{t_0} \int_{\RR} p_{t-s} (x-y)^2 \varphi(W(s,y))^2dsdy\right]\\
 \leq & 
  \|f\|_\infty^2
   \E\left[  \int_0 ^{t_0/t} \int_{\RR} p_{1-s} (\frac x{\sqrt{t}}-y)^2 |W(s,y)|^{2\alpha}dsdy\right]\rightarrow 0.
 \end{align*}

 Now, we continue with the term $ B_t $ for which we will use the decomposition
 \begin{align*}
 W(s,y)&= W(s,y)- W(t_0,y) + W(t_0,y),\\
 &=: \widehat{W}(s-t_0,y) + W(t_0,y).
 \end{align*}
Then, we can write
 \[
 B_t=
 \widehat{\E} \left[   \exp \left (i\lambda  t^{-\frac { 3\alpha +1}4}    \int_{0} ^{t-t_0} \int_{\RR} p_{t-t_0-s} (x-y) \varphi(\widehat{W}(s,y) + W(t_0,y) )  \widehat{W}(ds,dy )  \right)  \right],
 \]
 where $\widehat{\E}$ denotes the mathematical expectation with respect to the two-parameter Wiener process $\widehat{W}$. By the same renormalization arguments as in \eqref{eq:1} leading to \eqref{ecu3}, this gives 
 \begin{align*}
 B_t& =
 \widehat{\E} \Bigg[   \exp  \Bigg (i\lambda   \left( \frac  {t-t_0} t  \right)^{\frac { 3\alpha +1}4}    \int_{0} ^{1} \int_{\RR} p_{1-s} (\frac x{\sqrt{ t-t_0}}-y)
 F(t,s,y)   \widehat{W}(ds,dy )  \Bigg)  \Bigg].
 \end{align*} 
 Here,
 \begin{align*}
F(t,s,y)&:=
f((t-t_0)^{3/4}(\widehat{W}(s,y) +(t-t_0) ^{-\frac 34} W(t_0,
\sqrt{ t-t_0}y))) \\
&\quad\times
 |\widehat{W}(s,y) +(t-t_0) ^{-\frac 34} W(t_0,
 \sqrt{ t-t_0}y)|^\alpha .
 \end{align*}
 As $t\rightarrow \infty$ and $\lim _{x  \rightarrow \pm\infty} |x|^{-\alpha}\varphi(x) =c_\pm$, then $B_t$ converges almost surely to 
 \[
 \widehat{\E} \left[   \exp  \left (i\lambda       \int_{0} ^{1} \int_{\RR} p_{1-s} (y) F_\alpha(\widehat{W}(s,y))  \widehat{W}(ds,dy )  \right)  \right],
 \]
 where $ F_\alpha(x)=c_+\mathbf{1}_{x>0}+c_-\mathbf{1}_{x<0} $. Then the above formula is the characteristic function of { $X$}. As a consequence, for every bounded $\mathcal{F}_{t_0}$ measurable random variable $G$ we  obtain
 \[
 \lim_{t\rightarrow \infty} \E[ G e^{i\lambda  t^{-\frac { 3\alpha +1}4}    u(t,x)}]
 = \E[G] \E[ e^{i\lambda X}].
 \]
 
 This can be extended to any bounded random variable $G$ measurable with respect to the two-parameter Wiener process $W$ and this provides the desired stable convergence in the sense of Definition \ref{d:stable}.  \hfill $\square$

 \medskip
     \noindent
     For the proof of Theorem \ref{thm2}, we need the following lemma on the existence of the weighted local time  $L_{0,r}$.    
        
  \begin{lemma}  \label{lem1} For any $r\in [0,1]$, 
  the limit in $L^2(\Omega)$, as  $\ep$ tends to zero, of
   \[
   L_{0,r}^\varepsilon:= \int_0^r \int_{\RR}  p^2_{1-s} (y ) p_\varepsilon(W(s, y)) dyds,
 \]
 exists and the limit random variable  will be denoted by
   $L_{0,r}:=\int_0^r \int_{\RR}  p^2_{1-s} (y ) \delta_0(W(s, y)) dyds$.  
     \end{lemma}
     \begin{proof}
     Using the inverse Fourier transform formula for a Gaussian law, we have
     \[
      L^\varepsilon_{0,r}=\frac 1{2\pi} \int_0^r \int_{\RR^2}  p^2_{1-s} (y )  e^{ i\xi W(s,y) - \frac {\ep^2} 2 \xi^2}  d\xi dyds.
      \]
      Therefore,
      \begin{align*}
      \E( L^\varepsilon_{0,r} L^{\varepsilon'}_{0,r}) &=(2\pi)^{-2} \int_0^r \int_0^r \int_{\RR^4}  p^2_{1-s} (y ) p^2_{1-s'} (y')  \\
      &\qquad \times  \E\left( e^{ i\xi W(s,y) - \frac {\ep^2} 2 \xi'^2-i\xi' W(s',y') - \frac {\ep'^2} 2 \xi'^2 
      }  \right) d\xi dy d\xi' dy'\ ds'ds \\
      &=(2\pi)^{-2}   \int_0^r \int_0^r \int_{\RR^4}  p^2_{1-s} (y ) p^2_{1-s'} (y') e^{ - \frac {\ep^2} 2 \xi^2  -\frac {\ep'^2} 2 \xi'^2 } \\
      &\qquad \times  e^{ -\frac 12   \E ( |\xi W(s,y) - \xi'W(s',y')|^2)}
         d\xi dy d\xi' dy'\ ds'ds.
      \end{align*}
      As $\ep$ and $\ep'$ tend to zero we obtain the limit
      \[
  I:=    \int_0^1 \int_0^1 \int_{\RR^2}  p^2_{1-s} (y ) p^2_{1-s'} (y')  
 f(s,y,s',y')
         dy dy'\ ds'ds,
         \]
         where  $f(s,y,s',y')$ is the density at $(0,0)$ of the random vector $(W(s,y), W(s',y'))$.
         We claim that $I<\infty$. Indeed, first notice that
           $f(s,y,s',y')$ is bounded by
         \begin{align*}
     &(2\pi)^{-1}    \left ( ss' |y| |y'| - (s\wedge s')^2 ( |y| \wedge |y'| )^2 \mathbf{1}_{\{yy'>0\}}\right ) ^{-\frac 12}\\
     =&(2\pi)^{-1}   \left[ (s\wedge s') ( |y| \wedge |y'| ) \left ( (s\vee s') (|y|\vee |y'| )- (s\wedge s') ( |y| \wedge |y'|)  \mathbf{1}_{\{yy'>0\}}\right )  \right]^{-\frac 12}.
         \end{align*}
     To show that $I<\infty$, it suffices to consider the integral over the set $\{ yy'>0 \}$, because the integral over  $\{ yy' \le 0\}$ is clearly finite. By symmetry, we only need to show that the integral
     \[
   J:=  \int_{0<s<s'<1} \int_{ 0<y<y'<\infty} (1-s)^{-1} (1-s')^{-1}  (sy(s'y'-sy))^{-\frac 12} e^{-\frac {y^2} {1-s} -\frac{ y'^2} { 1-s'}} dy dy' ds ds' 
     \]
     is finite. With the change of variables $sy=z$ and $s'y'=z'$ we obtain
       \[
   J\le  \int_{0<s<s'<1} \int_{ 0<z<z'<\infty} [(1-s) (1-s')ss']^{-1}  (z(z'-z))^{-\frac 12} e^{-\frac {z^2} {s^2(1-s)} -\frac{ z'^2} { s'^2(1-s')}} dz dz' ds ds'.
     \]
     Fix $\delta>0$ and let $K_\delta = \sup_{s\in [0,1],z>0} \left( \frac {z^2}  { s^2(1-s)}   \right)^\delta e^{-\frac {z^2} {s^2(1-s)} }$. Then,
     \begin{align*}
     J & \le K_\delta \int_{0<s<s'<1} \int_{ 0<z<z'<\infty} (1-s)^{-1+\delta} (1-s')^{-1} s^{-1+2\delta}  (s')^{-1}  z^{-\frac 12-2\delta} (z'-z)^{-\frac 12} 
    \\
    & \times e^{-\frac{ z'^2} { s'^2(1-s')}}   dz dz' ds ds'.
     \end{align*}
      Integrating first in $z$ and later in $z'$ it is easy to show that the above integral is finite if $\delta<\frac 14$. This allows us to conclude the proof of  the lemma. 
                  \end{proof}

\medskip
\noindent
 {\it Proof of Theorem \ref{thm2}:}  
 Consider the random variable $\widetilde{u}(t,x)$ defined in (\ref{ecu3}). We can put $ \widetilde{u}(t,x)=  t^{-\frac 38}  M_t(1,x)$, where for $r\in [0,1]$, 
 \[
 M_t(r,x) :=  t^{\frac 38}   \int_0^r \int_{\RR}  p_{1-s} (\frac {x}{\sqrt{t}} -y )  \varphi( t^{3/4} W(s, y)) W(ds,dy).
 \]
 Then $\{M_t(\cdot, x), t\ge 0 \}$ is a family of continuous martingales in the time interval $[0,1]$.   We will find the limit as $t\rightarrow \infty$ of the quadratic variation of these martingales. We have
 \begin{align*}
 \langle M_t(\cdot,x) \rangle_{r} &= t^{\frac 34}  \int_0^r \int_\RR p^2_{1-s} (\frac {x}{\sqrt{t}} -y )   \varphi^2( t^{\frac 34} W(s, y))  dsdy.
  \end{align*}
  The proof of the theorem will be done in several steps.
  
  \medskip
  \noindent
  {\it Step 1.}  \quad 
  In this step we prove that 
  $\langle M_t(\cdot,x)  \rangle_{r}$ converges in $L^1(\Omega) $ to  the weighted local time $L_{0,r}$. First, we claim that
  \begin{equation} \label{ecu6}
  \lim_{t\rightarrow \infty} \E \left( \left|  \langle M_t(\cdot,x) \rangle_{r}  - t^{\frac 34}  \int_0^r \int_\RR p^2_{1-s}  (y) \varphi^2( t^{\frac 34} W(s, y))  dsdy \right|  \right)=0.
  \end{equation}
  This follows from the fact that
  \[
  \E\left(  t^{\frac 34}\varphi^2( t^{\frac 34} W(s, y))   \right) \le  \| \varphi\|_{L^2(\RR)}^2 (2\pi  s |y|)^{-\frac 12}
  \]
  and
  \[
  \lim_{t\rightarrow \infty}   \int_0^r \int_\RR  \left| p^2_{1-s} (\frac {x}{\sqrt{t}} -y ) -p^2_{1-s}  (y) \right|     \frac  1{ \sqrt{s |y|}} dyds =0.
  \]

On the other hand, for any fixed $t$, we have
\begin{equation} \label{ecu5}
\lim_{\ep \rightarrow 0} J_{\ep,t} =0,
\end{equation}
where
\begin{align*}
J_{\ep,t} &=
\E \Big( \Big|  t^{\frac 34}  \int_0^r \int_\RR p^2_{1-s}  (y) \varphi^2( t^{\frac 34} W(s, y))  dsdy\\
&-  \int_0^r \int_\RR p^2_{1-s}  (y)  \int_{\RR}  \varphi^2(\xi)   p_{\ep t^{-\frac 32} }( W(s, y) - t^{-\frac 34}\xi)d\xi   dsdy \Big|  \Big).
\end{align*}
To show (\ref{ecu5}), notice first  that
 \begin{align*}
&      \int_0^r \int_\RR p^2_{1-s}  (y)  \int_{\RR}  \varphi^2(\xi)   p_{\ep t^{-\frac 32}}( W(s, y) - t^{-\frac 34}\xi)d\xi   dsdy 
\\
&\qquad = t^{\frac 34}  \int_0^r \int_\RR p^2_{1-s}  (y)  \int_{\RR}  \varphi^2(\xi)   p_\ep( t^{\frac 34} W(s, y) -\xi)d\xi   dsdy \\
&\qquad =   t^{\frac 34}  \int_0^r \int_\RR p^2_{1-s}  (y) ( \varphi^2 * p_\ep) (t^{\frac 34} W(s, y))  dsdy.
  \end{align*}
  Therefore
  \begin{align*}
  J_{\ep,t} & = \E \left( \left|  t^{\frac 34}  \int_0^r \int_\RR p^2_{1-s}  (y) (\varphi^2-\varphi^2 * p_\ep) ( t^{\frac 34} W(s, y))  dsdy \right| \right)\\
  & \le t^{\frac 34}  \int_0^r \int_\RR p^2_{1-s}  (y)  \E(|  (\varphi^2-\varphi^2 * p_\ep) ( t^{\frac 34} W(s, y))  |) dsdy \\
  &\le \| \varphi^2 - \varphi^2 * p_\ep\|_{L^1(\RR)}    \int_0^r \int_\RR p^2_{1-s}  (y)  (2\pi s |y|) ^{-\frac 12} dy ds,
  \end{align*}
  which converges to zero as $\ep$ tends to zero because  $  \int_0^r \int_\RR p^2_{1-s}  (y)  (2\pi s |y|) ^{-\frac 12} dy ds <\infty$ and $\varphi \in L^2(\RR)$.

  We also  claim that
  \begin{equation}
 \lim_{t\rightarrow \infty}  \sup_{\ep >0}  I_{\ep,t}  =0, \label{ecu4}
\end{equation}
where
    \begin{align*}
  I_{\ep,t} &= \E \Big( \Big|  \int_0^r \int_{\RR^2} p^2_{1-s}  (y)    \varphi^2(\xi)   p_{\ep t^{-\frac 32}}( W(s, y) - t^{-\frac 34}\xi)d\xi   dy ds\\
  & \qquad  \qquad \qquad  - \| \varphi \|_{L^2(\RR)}^2 \int_0^r \int_\RR p^2_{1-s}  (y)    p_{\ep t^{-3/2} }( W(s, y) )  dy ds \Big|^2 \Big).
  \end{align*}
  To show (\ref{ecu4}), we write
   \begin{align*}
  I_{\ep,t} &= (2\pi)^{-2} \E \left( \left|  \int_0^r \int_{\RR^2} p^2_{1-s}  (y)    \varphi^2(\xi)   \int_{\RR}  e^{i \eta W(s,y)  - \frac {\ep^2 t^{-3}}2 \eta^2} ( e^{i\eta t^{-\frac 34} \xi} -1) 
  d\eta d\xi dyds \right|^2 \right) \\
  &= (2\pi)^{-2}  \int_{[0,r]^2}  \int_{\RR^4} p^2_{1-s}  (y) p^2_{1-s'}  (y')    \varphi^2(\xi)   \varphi^2(\xi')   \\
  & \qquad  \times \int_{\RR^2}  e^ {-\frac 12 \E( |\eta W(s,y)- \eta' W(s',y') |^2)}  e^{  - \frac {\ep^2 t^{-3}}2( \eta^2+ \eta'^2)}  \\
  & \qquad \times ( e^{i\eta t^{-\frac 34} \xi} -1)( e^{-i\eta' t^{-\frac 34} \xi'} -1)
  d\eta d\eta' d\xi d\xi'  dy dy'ds ds' ,
  \end{align*}
  which leads to the estimate
   \begin{align*}
  \sup_{\ep>0} I_{\ep,t} &\le    (2\pi)^{-2}  \int_{[0,r]^2}  \int_{\RR^4} p^2_{1-s}  (y) p^2_{1-s'}  (y')    \varphi^2(\xi)   \varphi^2(\xi')   \\
  & \qquad  \times \int_{\RR^2}  e^ {-\frac 12 \E( |\eta W(s,y)- \eta' W(s',y') |^2)}     \left(| \eta \xi \eta' \xi' | ^\beta t^{-\frac 34 \beta} \wedge 4\right)
  d\eta d\eta' d\xi d\xi'  dy dy'ds ds' ,
  \end{align*}
  for any $\beta \in [0,1]$. Then, by the dominated convergence theorem,  the limit (\ref{ecu4}) follows from
  \[
  \int_{[0,r]^2}  \int_{\RR^4} p^2_{1-s}  (y) p^2_{1-s'}  (y')      \int_{\RR^2}  e^ {-\frac 12 \E( |\eta W(s,y)- \eta' W(s',y') |^2)}    
  d\eta d\eta'  dy dy'ds ds'  <\infty,
  \]
  which follows from the fact that   $(2\pi)^{-2} \int_{\RR^2}  e^ {-\frac 12 \E( |\eta W(s,y)- \eta' W(s',y') |^2)}    
  d\eta d\eta' $ is the density at $(0,0)$ of the random vector $(W(s,y), W(s',y'))$ (see the proof of Lemma 2.1).
  
 By Lemma \ref{lem1}, $ \int_0^r \int_\RR p^2_{1-s}  (y)    p_{\ep }( W(s, y) )  dy ds$ converges in $L^2(\Omega)$ as $t\rightarrow \infty$ to the
 weighted local time $L_{0,r}$. As a consequence, from (\ref{ecu6}),  (\ref{ecu5}), (\ref{ecu4}) and Lemma  \ref{lem1} we deduce that 
   $\langle M_t(\cdot,x)  \rangle_{r}$ converges in $L^1(\Omega) $ to  the weighted local time $L_{0,r}=\int_0^r \int_{\RR}  p^2_{1-s} (y ) \delta_0(W(s, y)) dyds$.

   \medskip
   \noindent
   {\it Step 2.}  \quad  Fix an orthonormal basis $\{e_i, i\ge 1\}$ of $L^2(\RR)$ formed by bounded functions and consider the martingales
   \[
   M^i(r)= \int_0^r \int_{\RR} e_i(y) W(ds,dy), \quad r\in [0,1].
   \]
   We claim that the joint quadratic variation $\langle M_t(\cdot, x), M^i \rangle_r$ converges to zero in $L^1(\Omega)$ as $t\rightarrow \infty$.
  Indeed,
\[
  \langle M_t(\cdot, x), M^i \rangle_ r= t^{\frac 38}  \int_0^r  \int_{\RR} p_{1-s} ( \frac x {\sqrt{t}} -y) e_i(y) \varphi( t^{\frac 34} W(s,y)) dyds.
\]
Then, using the fact that $\varphi \in L^p(\RR)$ for some $p<2$, we can write
\begin{align*}
&\E (|  \langle M_t(\cdot, x), M^i \rangle_ r|) \le   t^{\frac 38}  \int_0^r  \int_{\RR} p_{1-s} ( \frac x {\sqrt{t}} -y) |e_i(y) | \E(|\varphi( t^{\frac 34} W(s,y)) |)dyds\\
&\qquad \le  t^{\frac 38}  \int_0^r  \int_{\RR} p_{1-s} ( \frac x {\sqrt{t}} -y)| e_i(y)|  \int_{\RR}  |\varphi(z)|  \frac 1{ \sqrt{ 2\pi t^{\frac 32} s|y|} }  \exp\left(-\frac{z^2} {2 t^{\frac 32} s|y|}\right) dzdyds
\\
& \qquad \le t^{\frac 38- \frac 3{ 4p}}   \| e_i \|_{\infty}  \| \varphi\|_{L^p(\RR)}  \int_0^r  \int_{\RR} p_{1-s} ( \frac x {\sqrt{t}} -y)(s|y|)^{- \frac 1{2p}} dyds,
  \end{align*}
  where $\frac 1p + \frac 1q =1$. Then, the claim follows because  $\frac 38-  \frac 3{ 4p} <0$  for $ p\in (1,2) $ and 
 \[
  \int_0^r  \int_{\RR} p_{1-s} ( \frac x {\sqrt{t}} -y)(s|y|)^{-\frac 1{2p}} dyds <\infty.
  \]
  
   \medskip
   \noindent
   {\it Step 3.}  \quad  Given a sequence $t_n \uparrow \infty$, set  $M^n_{0,r} = M_{t_n}(r,x)$ and $M^n_{i,r}= M^i(r)$ for $i \ge 1$.   These martingales, after possibly enlarging the probability space (in particular extending the definition of the martingales for $ r>1 $), possess  Dambis-Dubins-Schwarz Brownian motions $\beta^n_i$, 
   such that
   \[
   M^n_{0,r } = \beta^n_{0,\langle M^n_0 \rangle_r}
   \]
   and
   \[
   M^n_{i,r} = \beta^n_{i, r \int_{\RR} e_i(y)^2 dy}, \quad i \ge 1.
   \]
We have proved in Step 2 that  $ \sup_{r\in [0,1] } |\langle  M^n_i, M^n_0 \rangle_r|  \rightarrow 0$ in probability as $n\rightarrow \infty$. Moreover, it is clear that for any $1\le i <j$,  $   \langle  M^n_i, M^n_j \rangle_ r =0$. Then, by the asymptotic Ray-Knight theorem \cite{PY86}, we conclude that the Brownian motions $\beta^n_{i,y}$, $i\ge 0$,  converge in law to a family of independent Brownian motions $\beta_{i,y}$, $i\ge 0$. Together with Step 1, we obtain that $ M_{t_n} (r,x)$ converges weakly as
$n\rightarrow \infty$ to $\beta_{0, L_{0,r} \| \varphi \|^2_{L^2(\RR)} }$, where the Brownian motion $\beta_0$ is independent of the stochastic integrals $\{ \int_0^r e_i(y) W(ds,dy), r\in [0,1], i\ge 1\}$, that is, $\beta_0$ is independent of the white noise $W$ on $[0,1] \times\RR$.

   Thus, mutatis mutandis, we have proved the convergence  in law of  $(\widehat{W}, t^{\frac 38}  \widetilde{u}(t,x))$ to  $(\widehat{W}, \beta_{0, L_{0,1} \| \varphi \|^2_{L^2(\RR)} })$ as
   $t\rightarrow \infty$, where  $L_{0,1}=\int_0^1 \int_{\RR}  p^2_{1-s} (y ) \delta_0(\widehat{W}(s, y)) dyds$  and $\beta_0$ is independent of $\widehat{W}$.
   It remains to show the independence of $(\beta_0, \widehat{W})$ and $W$. For this we use the method of characteristic functions as in the proof of Theorem \ref{thm1b}.
   
      \medskip
   \noindent
   {\it Step 4.}  \quad Fix $\lambda \in \R$ and $t_0 \ge 0$. We follow a similar argument as in the proof of Theorem \ref{thm1b}. In fact, we can write
   \begin{align*}
   \E \left[ e^{i\lambda  t^{\frac 18} u(t,x)} \right] &=e^{i\lambda  t^{\frac 1 8 }     \int_0 ^{t_0} \int_{\RR} p_{t-s} (x-y)  \varphi(W(s,y) ) W(ds,dy) } \\
 & \times   \E \left[   e^{i\lambda  t^{\frac  18}    \int_{t_0} ^{t} \int_{\RR} p_{t-s} (x-y)   \varphi( W(s,y)) W(ds,dy)  } | \mathcal{F}_{t_0} \right]\\
 & =: e^{i\lambda A_t} \times B_t.
 \end{align*}
 As before, it is easy to show that $\lim_{t\rightarrow \infty}  A_t =0$ in $L^2(\Omega)$. On the other hand,  with the decomposition
 \[
 W(s,y)= W(s,y)- W(t_0,t) + W(t_0,y),
 \]
 for the term $B_t$, we can write
\[
B_t=
    \widehat{\E} \left[   \exp \left (i\lambda  t^{\frac  18}    \int_{0} ^{t-t_0} \int_{\RR} p_{t-t_0-s} (x-y)  \varphi( \widehat{W}(s,y) + W(t_0,y) ) \widehat{W}(ds,dy )  \right)  \right],
\]
where $\widehat{\E}$ denotes the mathematical expectation with respect to the two-parameter Wiener process $\widehat{W}$ defined by
$\widehat{W}(s,y)= W(s+ t_0,y) -W(s,y)$. 
 By the same arguments as before, this leads to
\begin{align*}
B_t& =
    \widehat{\E} \Bigg[   \exp \Bigg(  (i\lambda     { t^{\frac 18}}  (t-t_0)^{\frac 14}     \int_{0} ^{1} \int_{\RR} p_{1-s} (\frac x{\sqrt{t-t_0}}-y)\\
    & \qquad \times   \varphi( \widehat{W}((t-t_0)s, \sqrt{t-t_0}y) +(t-t_0) ^{-\frac 34} W(t_0,
      \sqrt{ t-t_0}y) )  \widehat{W}(ds,dy )  \Bigg)  \Bigg].
\end{align*} 
As $t\rightarrow \infty$, $B_t$ converges almost surely to
\[
\widehat{\E} \left[   \exp  \left (i\lambda      \beta_{0, L_{0,1} \| \varphi \|^2_{L^2(\RR)} } \right)  \right].
 \]
As a consequence, for every bounded and $\mathcal{F}_{t_0}$-measurable random variable $G$ we  obtain
 \[
 \lim_{t\rightarrow \infty} \E[ G \exp \left(i\lambda   t^{\frac 18} u(t,x) \right)]
= \E[G] \E\left[  \exp \left( i\lambda  \beta_{0, L_{0,1} \| \varphi \|^2_{L^2(\RR)} }  \right) \right].
 \]
 This completes the proof. \hfill $\square$

 \section{Large times and space averages}
 
 The asymptotic behavior of the spacial averages  $\int_{-R} ^R u(t,x) dx$ as $R\rightarrow \infty$ has been recently  studied in  the references \cite{DeNuZh,HuNuVi,HuNuViZh}. In these papers $u(t,x)$ is the solution to a stochastic partial differential equation with initial condition $u(0,x)=1$ and a Lipschitz nonlinearity $\sigma(u)$. The solution process is stationary in $x\in \R$ and the limit is Gaussian with a proper normalization.  In the case considered here, the lack of stationarity creates different limit behaviors.  In order to achieve a limit, we will consider the case where both $R$ and $t$ tend to infnity. 
 
 Set
 \[
 u_R(t)= \int_{-R} ^R \int_0^t \int_{\R}  p_{t-s} (x-y) \varphi(W(s,y))W(ds,dy) dx.
 \]
 As before, $u_R(t)$ has the same law as
  \[
\widetilde{ u}_R(t)= t^{\frac 14}  \int_{-R} ^R \int_0^1 \int_{\R}  p_{1-s} (\frac x{\sqrt{t}}-y) \varphi(t^{\frac 34} W(s,y))W(ds,dy) dx.
 \]
Consider first  the case where $\varphi$ is an homogeneous function. 
\begin{theorem}  \label{thm1a}  Suppose that $\varphi(x) = |x| ^\alpha$ for some $\alpha >0$.  Suppose that $t_R \rightarrow \infty$  as $R\rightarrow \infty$. Then,
with    $\widehat{W}$ is a two-parameter Wiener process independent of $W$, the following stable convergences hold true:
\begin{itemize}
\item[(i)] If $\frac R  {\sqrt{t_R}} \rightarrow c$, with $c\in (0,\infty)$, 
\[
t_R^{-\frac  34 (\alpha+1)} u(t_R)   \stackrel  {\rm stably}  \longrightarrow   \int_{-c} ^c   \int_0^1   \int_{\R} p_{1-s} (x-y) |\widehat{W}(s, y))|^\alpha \widehat{W}(ds,dy)dx.
\]
\item[(ii)] If $\frac R  {\sqrt{t_R}} \rightarrow 0$,  
\[
R^{-1} t_R^{-\frac  {3\alpha +1} 4} u(t_R)   \stackrel  {\rm stably}  \longrightarrow   2  \int_0^1   \int_{\R} p_{1-s} (y) |\widehat{W}(s, y))|^\alpha \widehat{W}(ds,dy)dx.
\]
\item[(iii)] If $\frac R  {\sqrt{t_R}} \rightarrow \infty$,  
\[
R ^{-\frac{ \alpha +1}2}  t_R^{-\frac  {\alpha +1}2}  u(t_R)   \stackrel  {\rm stably}  \longrightarrow  \int_{-1} ^{ 1} \int_0^1   |\widehat{W}(s,y)| ^{\alpha}  \widehat{W}(ds,dy).
\]
\end{itemize}
 \end{theorem}
 
  \begin{proof} We have, with the change of variable $\frac x{\sqrt{t_R}}  \rightarrow x $,
   \[
\widetilde{ u}_R(t_R)= t_R^{\frac  {3\alpha+1} 4+\frac 12}  \int_{-R/\sqrt{t_R}} ^{ R/\sqrt{t_R}} \int_0^1\int_{\R}  p_{1-s} (x-y)  |W(s,y)| ^\alpha W(ds,dy) dx,
 \]
 and (i)  follows by letting $R\rightarrow \infty$. If $\frac R  {\sqrt{t_R}} \rightarrow 0$,  with the change of variable  $x  \rightarrow Rx $, we can write
    \begin{equation} \label{eq1}
\widetilde{ u}_R(t_R)=R t_R^{\frac  {3\alpha+1} 4}  \int_{-1} ^{ 1} \int_0^1\int_{\R}  p_{1-s} (x\frac R{\sqrt{t_R}}-y)  |W(s,y)| ^\alpha W(ds,dy) dx,
 \end{equation}
 which implies (ii). The proof of (iii) is more involved. Making the change of variable $y \rightarrow  y \frac R{\sqrt{t_R}}$ in  (\ref{eq1}) yields
\begin{align*}
\widetilde{ u}_R(t_R)&=R ^{\frac{ \alpha +3}2}  t_R^{\frac  \alpha 2} 
\int_{-1} ^{ 1} \int_0^1\int_{\R}  p_{1-s} (\frac R{\sqrt{t_R}}(x-y))  |W(s,y)| ^\alpha W(ds,dy) dx \\
& =R ^{\frac{ \alpha +1}2}  t_R^{\frac { \alpha +1}2} 
\int_{-1} ^{ 1} \int_0^1\int_{\R}  p_{\frac {t_R}{R^2}(1-s)} (x-y)  |W(s,y)| ^\alpha W(ds,dy) dx.
\end{align*}
Finally the stochastic integral
\[
\int_{-1} ^{ 1} \int_0^1\int_{\R}  p_{\frac {t_R}{R^2}(1-s)} (x-y)  |W(s,y)| ^{\alpha}  W(ds,dy)dx
\]
converges in $L^2(\Omega)$ as $R\rightarrow \infty$ to 
\[
\int_{-1} ^{ 1} \int_0^1   |W(s,y)| ^{\alpha}  W(ds,dy).
\]
The stable character of the convergence can be proved by the same arguments, based on the conditional characteristic function, as in the proof of Theorem \ref{thm1b}.
 \end{proof}
 For a function which satisfies integrability  conditions with respect to the Lebesgue measure, we have the following result.
 
 \begin{theorem}  \label{thm11}  Suppose that $\varphi(x) \in L^2(\R)$.  Suppose that $t_R \rightarrow \infty$  as $R\rightarrow \infty$. Then,
with $Z$ a $N(0,1)$ random variable and     $\widehat{W}$  an independent two-parameter Wiener process such that $(Z, \widehat{W})$ are   independent of $W$, the following stable convergences hold true:
\begin{itemize}
\item[(i)] If $\frac R  {\sqrt{t_R}} \rightarrow c$, with $c\in (0,\infty)$, 
\[
t_R^{-\frac  38} u(t_R)   \stackrel  {\rm stably}  \longrightarrow   Z \left( \|\varphi \|^2_{L^2(\R)}\int_0^1 \int_{\R} \left( \int_{-c} ^c p_{1-s} (x-y) dx \right)^2 \delta_0(W(s,y)) dyds \right)^{\frac 12}.
\]
\item[(ii)] If $\frac R  {\sqrt{t_R}} \rightarrow 0$,  
\[
R^{-1} t_R^ {\frac 12} u(t_R)   \stackrel  {\rm stably}  \longrightarrow     Z \left( 2\|  \varphi \|^2_{L^2(\R)}\int_0^1 \int_{\R}  p^2_{1-s} (y)   \delta_0(W(s,y)) dyds \right)^{\frac 12}.
\]
\item[(iii)] If $\frac R  {\sqrt{t_R}} \rightarrow \infty$,  
\[
R ^{-\frac 12}  t_R^{\frac 14}  u(t_R)   \stackrel  {\rm stably}  \longrightarrow   Z \left(2 \|  \varphi \|^2_{L^2(\R)}\int_0^1 \int_{-1}^1  p^2_{1-s} (y)   \delta_0(W(s,y)) dyds \right)^{\frac 12}.
\]
\end{itemize}
 \end{theorem}
 
 \begin{proof}  Let us prove first the case (i). We have, with the change of variable $\frac x{\sqrt{t_R}}  \rightarrow x $,
   \[
\widetilde{ u}_R(t_R)= t_R^{\frac  34}  \int_{-R/\sqrt{t_R}} ^{ R/\sqrt{t_R}} \int_0^1\int_{\R}  p_{1-s} (x-y)   \varphi( t_R^{\frac 34} W(s,y)) W(ds,dy) dx.
 \]
 Consider the family of martingales
 \[
 M_R(\cdot, x)= t_R^{\frac  38}  \int_{-R/\sqrt{t_R}} ^{ R/\sqrt{t_R}} \int_0^\cdot\int_{\R}  p_{1-s} (x-y)   \varphi( t_R^{\frac 34} W(s,y)) W(ds,dy) dx,
 \]
 $r\in [0,1]$.  We can write
 \[
\langle M_R(\cdot, x )\rangle_r= t_R^{\frac  34} \int_{ [ -\frac R{\sqrt{t_R}}, \frac R{\sqrt{t_R}} ]^2} \int_0^r\int_{\R}  p_{ 1-s} (x-y) p_{1-s} (x'-y) 
\varphi^2( t_R^{\frac 34} W(s,y))   dydsdx dx'.
\] 
Then,  as in the proof of Theorem  \ref{thm2}, we can show that $ \langle M_R(\cdot, x )\rangle_r$ converges in $L^1(\Omega)$ as $R\rightarrow \infty$ to the weighted local time
\[
\| \varphi \|^2_{L^2(\R)}\int_0^r \int_{\R} \left( \int_{-c} ^c p_{1-s} (x-y) dx \right)^2 \delta_0(W(s,y)) dyds.
\]
 This  completes the proof of (i). 
 
  If $\frac R  {\sqrt{t_R}} \rightarrow 0$,  with the change of variable  $x  \rightarrow Rx $, we can write
    \begin{equation} \label{eq2}
\widetilde{ u}_R(t_R)=R t_R^{\frac  1 4}  \int_{-1} ^{ 1} \int_0^1\int_{\R}  p_{1-s} (x\frac R{\sqrt{t_R}}-y)  \varphi( t_R^{\frac 34} W(s,y)) W(ds,dy) dx.
 \end{equation}
 As before, the stochastic integral
 \[
  t_R^{\frac  3 4}  \int_{-1} ^{ 1} \int_0^1\int_{\R}  p_{1-s} (x\frac R{\sqrt{t_R}}-y)  \varphi( t_R^{\frac 34} W(s,y)) W(ds,dy) dx
  \]
  converges in law to
  \[
  Z \left( 2\|  \varphi \|^2_{L^2(\R)}\int_0^1 \int_{\R}  p^2_{1-s} (y)   \delta_0(W(s,y)) dyds \right)^{\frac 12},
  \]
  which implies (ii).
To show (iii), we make the change of variable $y \rightarrow  y \frac R{\sqrt{t_R}}$ in  (\ref{eq2}) to get
\begin{align*}
\widetilde{ u}_R(t_R)&=R ^{\frac 32  }
\int_{-1} ^{ 1} \int_0^1\int_{\R}  p_{1-s} (\frac R{\sqrt{t_R}}(x-y))  \varphi(t_R^{\frac 34} W(s,y)) W(ds,dy) dx \\
& =R ^{\frac 12}  t_R^{\frac 12} 
\int_{-1} ^{ 1} \int_0^1\int_{\R}  p_{\frac {t_R}{R^2}(1-s)} (x-y)  \varphi(t_R^{\frac 34} W(s,y)) W(ds,dy) dx.
\end{align*}
Finally the stochastic integral
\[
t_R^{\frac 34}\int_{-1} ^{ 1}  \int_0^1\int_{\R}  p_{\frac {t_R}{R^2}(1-s)} (x-y)  \varphi(t_R^{\frac 34} W(s,y)) W(ds,dy) dx
\]
converges in  law as $R\rightarrow \infty$ to 
\[
  Z \left(2 \|  \varphi \|^2_{L^2(\R)}\int_0^1 \int_{-1}^1  p^2_{1-s} (y)   \delta_0(W(s,y)) dyds \right)^{\frac 12}.
\]
The stable character of the convergence can be proved by the same arguments, based on the conditional characteristic function, as in the proof of Theorem \ref{thm2}.
 \end{proof}

\section{Case of a nonlinear coefficient $\sigma$}
In this section we  discuss the case of  a nonlinear stochastic heat equation
\begin{equation} \label{ecu7}
\frac {\partial u }{\partial t} = \frac 12 \frac {\partial ^2u }{\partial x^2} + \sigma(u)  \frac { \partial ^2 W} {\partial t\partial x}, \quad x\in \RR, t\ge 0,
\end{equation}
with initial condition $u(0,z)=1$, where $\sigma: \R \rightarrow \R$ is a Lipschitz function.
 The mild solution to equation (\ref{ecu7}) is given by
\[
u(t,x)= 1+  \int_0^t \int_{\RR}  p_{t-s} (x-y) \sigma(u(s,y)) W(ds,dy).
\]
We are interested in the asymptotic behavior of $u(t,x)$ as $t$ tends to infinity. As before we consider different cases:

\medskip
\noindent
{\it Case 1.}  \quad Suppose that $\sigma(u)=u$. In this case, the solution has a Wiener chaos expansion given by
\begin{align*}
u(t,x) &=1+  \sum_{n\geq 1} \int_{\R^{n}} \int_{\Delta_n(t) } \prod_{i=0}^{n-1} p_{s_i - s_{i+1}} ( x_i - x_{i+1} ) \,  \, W(ds_1, dx_1)\cdots W(ds_n, dx_n) \\
& =: 1 + \sum_{n\geq 1} I_n(f_{t,x,n} ) \,,
\end{align*}
with 
\[
 f_{t,x,n}(s_1, \dots, s_n, x_1, \dots, x_n) =  {\bf 1}_{\Delta_n(t)}(s_1,\dots, s_n) \prod_{i=0}^{n-1} p_{s_i - s_{i+1}}( x_i - x_{i+1} ) 
 \]
 and $\Delta_n(t) =\{ (s_1, \dots, s_n): 0<s_1 < \cdots < s_n <t\}$.
Here $I_n$ denotes the multiple stochastic integral of order $n$ with respect to the noise $W$.
If we consider the projection of $u(t,x)$ on a fixed Wiener chaos, we can write 
with the change of variables $s_i \rightarrow ts_i$  and $x_i \rightarrow \sqrt{t }y_i$, 
\begin{align*}
 I_n(f_{t,x,n} )&=   \int_{\Delta_n(1) }\int_{\RR^n} p_{t(1-s_1)}(\frac x{\sqrt{t}} -x_1) \\
 & \qquad \times \prod_{i=1}^{n-1}  p_{t(s_i-s_{i+1})} (x_i-x_{i+1})    W(tds_1,\sqrt{t}dx_1) \cdots W(tds_n,\sqrt{t}dx_n) \\
&=t^{-\frac n2}   \int_{\Delta_n(1) }\int_{\RR^n} p_{1-s_1}(\frac x{\sqrt{t}} -x_1) \\
 & \qquad \times \prod_{i=1}^{n-1}  p_{s_i-s_{i+1}} (x_i-x_{i+1})    W(tds_1,\sqrt{t}dx_1) \cdots W(tds_n,\sqrt{t}dx_n).
\end{align*}
By the scaling properties of the two-parameter Wiener process it follows that $ I_n(f_{t,x,n} )$ has the same law as
\begin{align*}
\widetilde{I}_n(f_{t,x,n} ) &:=t^{\frac {3n}4}   \int_{\Delta_n(1) }\int_{\RR^n} p_{1-s_1}(\frac x{\sqrt{t}} -x_1) \\
 & \qquad \times \prod_{i=1}^{n-1}  p_{s_i-s_{i+1}} (x_i-x_{i+1})    W(ds_1,dx_1) \cdots W(ds_n,dx_n).
\end{align*}
As a consequence, $t^{-\frac 34 n}  I_n(f_{t,x,n} )$ converges stably to
\[
 \int_{\Delta_n(1) }\int_{\RR^n}    \prod_{i=0}^{n-1}  p_{s_i-s_{i+1}} (x_i-x_{i+1})   \widehat{W}(ds_1,dx_1) \cdots \widehat{W}(ds_n,dx_n),
\]
where $\widehat{W}$ is a two-parameter Wiener process independent of $W$ and
with the convention $s_0=1$ and $x_0=0$.

Notice that the rate of convergence depends on the order of the Wiener chaos. This is consistent with the asymptotic behavior of $\log u(t,x)$, when $u(0,x)=\delta_0(x)$,
obtained by Amir, Corwin and Quastel in \cite{AmCoQu}.

\medskip
\noindent
{\it Case 2.}  \quad  When $\sigma$ is a Lipschitz function that belongs to $L^2(\R)$, the problem is much more involved and we will give here just some ideas on how to proceed. We can write
\[
u(t,x)=1+ \frac 1{\sqrt{t}}  \int_0^1 \int_{\RR} p_{1-s}(\frac x{\sqrt{t}} -y)  \sigma(u( ts, \sqrt{t} y)) W(tds ,\sqrt{t}dy).
\]
Furthermore,
\begin{align*}
u( ts, \sqrt{t} y) &= 1+ \int_0^{ts} \int_{\RR} p_{ts-r} (\sqrt{y} -z) \sigma(u(r,z)) W(dr,dz) \\
&=1+ \frac 1 {\sqrt{t}} \int_0^s \int_{\RR} p_{s-r} (y-z) \sigma( u(tr, \sqrt{t} z)) W(t dr, \sqrt{t} dz).
\end{align*}
By the scaling properties of the two-parameter Wiener process, as a function of $\widehat{W}$, $u( ts, \sqrt{t} y)$ has the same law as
\[
v^t(s,y)= 1+ t^{\frac 14} \int_0^s \int_{\RR} p_{s-r} (y-z) \sigma( v^t(r,z))) \widehat{W}(dr, dz).
\]
Therefore, $u(t,x)$ has the same law as
\[
\widetilde{u}(t,x)= 1+ t^{\frac 14}  \int_0^1 \int_{\RR} p_{1-s}  (\frac x{\sqrt{t}} -y)  \sigma(v^t( s,   y)) \widehat{W}(ds ,dy).
\]
Then,
\[
t^{-\frac 16} \widetilde{u}(t,x)= t^{-\frac 16} + t^{\frac 1{12}}  \int_0^1 \int_{\RR} p_{1-s}  (\frac x{\sqrt{t}} -y)  \sigma(v^t( s,   y)) \widehat{W}(ds ,dy).
\]
The quadratic variation of the martingale part of the above stochastic integral is
\begin{align*}
& t^{\frac 16} \int_0^1 \int_{\RR} p^2_{1-s}  (\frac x{\sqrt{t}} -y)  \sigma^2(1+ t^{\frac 16} Z^t( s,   y)) dsdy \\
&=  \int_0^1 \int_{\RR} p^2_{1-s}  (\frac x{\sqrt{t}} -y)  \int_{\RR} \sigma^2(\xi) \delta_0(   Z^t( s,   y) + t^{-\frac 16} -\xi t^{-\frac 16})  d\xi dsdy 
\end{align*}
where $Z^t(s,y)$ satisfies
\[
Z^t( s,   y)= t^{-\frac 16} + t^{\frac 1{12}} \int_0^s \int_{\RR} p_{s-r} (y-z) \sigma( 1+ t^{\frac 16} Z^t(r,z))) \widehat{W}(dr, dz).
\]
From these computations, we conjecture that  $t^{-\frac 16}$  is the right normalization and the limit would satisfy an equation involving a weighted local time of the solution.  However, proving these facts is a challenging problem not to be treated in this paper.


\begin{thebibliography}{9}
 
 \bibitem{AmCoQu}
 G. Amir, I. Corwin, J. Quastel. 
 Probability distribution of the free energy of the continuum directed random polymer in $1+1$ dimension.
 {\it Communications in Pure and Applied Mathematics}.

\bibitem{Cam}
S. Campese. A limit theorem for moments in space of the increments of Brownian local time. {\it Ann. Probab.}  {\bf 45}, no. 3,  1512-1542 (2017).
 
 \bibitem{DeNuZh}
 F. Delgado-Vences, D. Nualart and G. Zheng:
 A Central Limit Theorem for the stochastic wave equation with fractional noise. Preprint
 
 \bibitem{HuNuVi}
J. Huang, D. Nualart and L.  Viitasaari:  A central limit theorem for the stochastic heat equation. Preprint.

\bibitem{HuNuViZh}
J. Huang, D. Nualart, L.  Viitasaari and G. Zheng:
Gaussian fluctuations for the stochastic heat equation with colored noise. Preprint.

 \bibitem{JacSh}  J. Jacod  and  A. N. Shiryaev. \textit{Limit Theorems
for Stochastic Processes}. Springer, Berlin, 1987.
 
\bibitem{Nualart06}
D. Nualart.  {\it The Malliavin Calculus and Related Topics}, second edition.robability and Its Applications,  Springer-Verlag, Berlin, Heidelberg, 2006.

\bibitem{NuXu}
D. Nualart and F. Xu.  Central limit theorem for an additive functional of the fractional Brownian motion. {\it Ann. Probab.} {\bf 42}, no. 1,  169-302 (2014).

 \bibitem{PY86}
J. Pitman and M. Yor. Asymptotic laws of planar Brownian motion. {\it Ann. Probab.} {\bf 14}, no. 3, 733-779 (1986).
  

     
\end{thebibliography}
\end{document}